\documentclass[11pt, letterpaper]{amsart}

\usepackage{amsfonts}
\usepackage{amssymb}
\usepackage{graphicx}
\usepackage{amsmath}
\usepackage{latexsym}
\usepackage{amscd}
\usepackage{xypic}
\usepackage{mathrsfs}
\usepackage{enumitem} 
\usepackage{braket}
\usepackage{bm}
\usepackage{color}

\usepackage{imakeidx}
\makeindex[title=Index of Notation]

\usepackage{hyperref}

\allowdisplaybreaks 

\numberwithin{equation}{section}
\newtheorem{theorem}{Theorem}[section]
\newtheorem{lemma}[theorem]{Lemma}
\newtheorem{corollary}[theorem]{Corollary}

\newtheorem{proposition}[theorem]{Proposition}

\newtheorem*{theorem*}{Theorem}

\theoremstyle{definition}
\newtheorem{definition}[theorem]{Definition}

\theoremstyle{remark}
\newtheorem*{remark}{Remark}

\theoremstyle{remark}

\theoremstyle{definition}

\newcommand{\RR}{\mathbb{R}}
\renewcommand{\SS}{\mathbb{S}}

\newcommand{\ZZ}{\mathbb{Z}}

\newcommand{\cC}{\mathcal C}
\renewcommand{\cD}{\mathcal D}

\newcommand{\cF}{\mathcal F}

\renewcommand{\cH}{\mathcal H}

\newcommand{\cM}{\mathcal M}
\newcommand{\cN}{\mathcal N}

\renewcommand{\cR}{\mathcal R}
\newcommand{\cS}{\mathcal S}

\newcommand{\bH}{\mathbf{H}}

\newcommand{\bx}{\mathbf{x}}

\newcommand{\bOh}{\mathbf{0}}

\newcommand{\fF}{\mathfrak{F}}

\DeclareMathOperator{\supp}{supp}

\newcommand{\eps}{\varepsilon}

\DeclareMathOperator{\sing}{sing}
\DeclareMathOperator{\reg}{reg}

\title{Mean curvature flow with generic low-entropy initial data}

\author[Chodosh]{Otis Chodosh} 
\address{OC: Department of Mathematics, Bldg.\ 380, Stanford University, Stanford, CA 94305, USA}
\email{ochodosh@stanford.edu}

\author[Choi]{Kyeongsu Choi}
\address{KC: School of Mathematics, Korea Institute for Advanced Study, 85 Hoegiro, Dongdaemun-gu, Seoul 02455, Republic of Korea}
\email{choiks@kias.re.kr}
\author[Mantoulidis] {Christos Mantoulidis} 
\address{CM: Department of Mathematics, Rice University, Houston, TX 77005, USA}
\email{christos.mantoulidis@rice.edu}
\author[Schulze]{Felix Schulze}
\address{FS: Department of Mathematics, Zeeman Building, University of Warwick, Gibbet Hill Road, Coventry CV4 7AL,
UK}
\email{felix.schulze@warwick.ac.uk} 

\begin{document}

\begin{abstract}
We prove that sufficiently low-entropy closed hypersurfaces can be perturbed so that their mean curvature flow encounters only spherical and cylindrical singularities. Our theorem applies to all closed surfaces in $\RR^3$ with entropy $\leq 2$ and to all closed hypersurfaces in $\RR^4$ with entropy $\leq \lambda(\SS^1 \times \RR^2)$. When combined with recent work of Daniels-Holgate, this strengthens Bernstein--Wang's low-entropy Schoenflies-type theorem by relaxing the entropy bound to $\lambda(\SS^1 \times \RR^2)$.

Our techniques, based on a novel density drop argument, also lead to a new proof of generic regularity result for area-minimizing hypersurfaces in eight dimensions (due to Hardt--Simon and Smale).
\end{abstract}

\maketitle



\section{Introduction}

Mean curvature flow is the natural heat equation for submanifolds. A family of hypersurfaces $M(t) \subset \RR^{n+1}$ flows by mean curvature flow if 
\begin{equation} \label{eq:mcf}
\left(\tfrac{\partial}{\partial t} \bx \right)^{\perp} = \bH_{M(t)}(\bx),
\end{equation}
where $\bH_{M(t)}(\bx)$ denotes the mean curvature vector of $M(t)$ at $\bx$. When $M(0)$ is compact, mean curvature flow is guaranteed to become singular in finite time. Understanding the potential singularities is thus a fundamental problem. One approach to this issue is to study the flow in the generic case: a well-known conjecture of Huisken suggests that the singularities of a generic mean curvature flow should be as simple as possible, namely, spherical and cylindrical \cite[\#8]{Ilmanen:problems}. 

The main results of this note completely resolve Huisken's conjecture in three and four dimensions for low-entropy initial data  (see \eqref{eq:def-entropy} for the definition of entropy). Informally stated (see Corollaries \ref{coro:main-3d} and \ref{coro:main-4d} for precise statements) we prove the following results.
\begin{theorem}[Low-entropy generic flow in $\RR^3$, informal]\label{theo:R3-informal}
	If $M^2\subset\RR^3$ is a closed embedded surface with entropy $\lambda(M) \leq 2$ then there exist arbitrarily small $C^\infty$ graphs $M'$ over $M$ so that the mean curvature flow starting from $M'$ has only multiplicity-one spherical and cylindrical singularities. 
\end{theorem}
\begin{theorem}[Low-entropy generic flow in $\RR^4$, informal]\label{theo:R4-informal}
If $M^3\subset\RR^4$ is a closed embedded hypersurface with entropy $\lambda(M) \leq \lambda(\SS^1\times\RR^2)$ then there exist arbitrarily small $C^\infty$ graphs $M'$ over $M$ so that the mean curvature flow starting from $M'$ has only multiplicity-one spherical and cylindrical singularities. 
\end{theorem}

In an earlier version of this paper, we conjectured that Theorem \ref{theo:R4-informal} could be combined with a surgery construction to yield a strengthened version of Bernstein--Wang's low-entropy Schoenflies theorem \cite{BernsteinWang:schoenflies} (cf.\ Theorem \ref{theo:BW-schoenflies} below). This surgery construction has been recently carried out by Daniels-Holgate \cite{Daniels-Holgate} who showed that if a mean curvature flow has only spherical and neckpinch singularities, then one can construct a mean curvature flow with surgery. As such, combining these results leads to the following:

\begin{corollary}[Strengthened low-entropy Schoenflies-type theorem]\label{coro:impr-schoenflies}
If $M^{3}\subset \RR^{4}$ is an embedded $3$-sphere with entropy $\lambda(M) \leq \lambda(\SS^{1}\times \RR^{2})$ then $M$ is smoothly isotopic to the round $\SS^{3}$. 
\end{corollary}

See Sections \ref{subsec:entropy} and \ref{subsec:surgery} for an expanded discussion of this result. 



 
\subsection{Previous work on generic mean curvature flow}\label{subsec:previous-work}
Trailblazing work of Colding--Minicozzi demonstrated that spheres and cylinders are the only \emph{linearly stable} singularity models for mean curvature flow \cite{ColdingMinicozzi:generic}. In particular, the remaining singularity models are unstable so should not generically occur (as conjectured by Huisken). In a previous paper \cite{CCMS:generic1}, the authors introduced new methods to the study of generic mean curvature flow, proving that a large class of singularity models (specifically, singularities with tangent flows modeled on multiplicity one compact or asymptotically conical self-shrinkers) can be indeed avoided by a slight perturbation of the initial conditions. 

In particular, our previous work shows that for a generic initial surface in $\RR^3$, either the mean curvature flow has only spherical and cylindrical singularities or at the first singular time it has a tangent flow with a cylindrical end or higher multiplicity (both possibilities are conjectured not to happen). We refer the reader to the introduction to our previous article \cite{CCMS:generic1} for further discussion of generic mean curvature flows and related work. 

\subsubsection{Relationship between this paper and our previous work}   In \cite{CCMS:generic1}, we proved a classification of ancient one-sided flows (analogous to the minimal surface results of Hardt--Simon \cite{HardtSimon:foliation}; see Appendix \ref{app:HS} for further discussion) which led to a complete understanding of flows on either side of a neighborhood of a non-generic (compact or asymptotically conical) singularity. In particular, we showed that nearby flows to either side do not have such singularities nearby. 

In $\RR^3$, to understand generic mean curvature flow \emph{without} a low-entropy condition (in contrast with this note), one must work at the first non-generic time rather than globally in space-time. However, two serious issues arise when working this way. First, there is no partial regularity known for tangent flows past\footnote{At the first singular time, work of Ilmanen \cite{Ilmanen:singularities} and Wang \cite{Wang:ends-conical} show that the support of any tangent flow is a smooth self-shrinker with only conical/cylindrical ends.} the first singular time without a low-entropy bound. Second, the possibility that a small perturbation of the initial data increases the first singular time slightly without improving the flow in an effective way. To that end, in \cite{CCMS:generic1} we had to additionally prove that the nearby flows strictly decrease genus as they avoid the non-generic singularity.  This genus-loss property is crucial for tackling Huisken's conjecture in $\RR^3$ without a low-entropy condition and is a consequence of the classification of ancient one sided flows, as obtained in \cite{CCMS:generic1}. 

On the other hand, by including a low-entropy condition, here we are able to work \emph{globally} in space-time. This allows for significantly simplified arguments. In fact, the key observation of this paper is that in this setting one can completely avoid the classification of one-sided ancient flows and instead rely on a soft argument based on compactness and a new geometric property of non-generic shrinkers (see Proposition \ref{prop:geometric-generic}). We emphasize that a drawback of the methods used in this note as compared to our previous work is that the arguments used here give no indication as to the local dynamics near a non-generic singularity (such information was obtained in \cite{CCMS:generic1} near asymptotically conical and compact shrinkers; see also \cite{Colding-Minicozzi:dynamics,ChoiMantoulidis}).

\begin{remark}
After the first version of this paper (as well as our previous paper \cite{CCMS:generic1}) were posted, another approach to the generic perturbation of the initial data was pursued by Sun--Xue \cite{SunXue:closed,SunXue:AC}. This approach is in the spirit of local ODE dynamics, as suggested by the Colding--Minicozzi program, cf.\ \cite{Colding-Minicozzi:dynamics}. The analytic framework in \cite{SunXue:closed,SunXue:AC} has the interesting feature that non-one-sided perturbations are analyzed, but the applications are currently limited to locally perturbing away singularities that arise at the first singular time. Conversely, our geometric approach (first developed in \cite{CCMS:generic1}) is motivated by global results such as the ones stated in Theorems \ref{theo:R3-informal} and \ref{theo:R4-informal}. Of course, our approach also admits  localizations; see  Appendix \ref{app:local}.
\end{remark}

\subsection{Entropy} \label{subsec:entropy}
To state our main results, we first recall Colding--Minicozzi's definition \cite{ColdingMinicozzi:generic} of entropy of $M^n\subset \RR^{n+1}$:
\begin{equation}\label{eq:def-entropy}
\lambda(M) := \sup_{\substack{\bx_0\in\RR^{n+1}\\t_0>0}}\int_{M} (4\pi t_0)^{-\frac n 2} e^{-\frac{1}{4t_0} |\bx - \bx_0|^{2}} .
\end{equation}
By Huisken's monotonicity of Gaussian area, we see that $t\mapsto \lambda(M(t))$ is non-increasing  when $M(t)$ is flowing by mean curvature flow. 
A computation of Stone \cite{Stone} shows that the entropies of the self-shrinking cylinders $\SS^k(\sqrt{2k})\times \RR^{n-k}\subset \RR^{n+1}$ satisfy\footnote{Note that $\lambda(\SS^k(\sqrt{2k})\times \RR^{n-k}) = \lambda(\SS^k)$.}
\[
2 > \lambda(\SS^1)=\sqrt{\frac{2\pi}{e}} \approx 1.52  > \frac 32 > \lambda(\SS^2) = \frac{4}{e} \approx 1.47 > \dots > \lambda(\SS^n). 
\]
Several fundamental results have been obtained about hypersurfaces with sufficiently small entropy, starting with work of Colding--Ilmanen--Minicozzi--White \cite{ColdingMinicozziIlmanenWhite} who proved that the round sphere $\SS^n(\sqrt{2n})$ has minimal entropy among all closed self-shrinkers. This was extended by Bernstein--Wang \cite{BernsteinWang:1} who showed that the round sphere minimizes entropy among all closed hypersurfaces (see also \cite{Zhu:entropy,HershkovitsWhite:sharp-entropy}). Moreover, Bernstein--Wang have also proven \cite{BernsteinWang:TopologicalProperty} that the cylinder $\SS^1(\sqrt{2})\times \RR \subset \RR^3$ has second least entropy among all self-shrinkers in $\RR^3$ (their result crucially relies on Brendle's classification of genus zero self-shrinkers \cite{Brendle:genus0}). 

Subsequent work of Bernstein--Wang provides a robust picture of hypersurfaces with sufficiently small entropy \cite{BernsteinWang:topology-small-ent,BernsteinWang:hausdorff-stability,bernsteinWang:top-uniqueness-expanders} (see also \cite{BernsteinSWang:disconnect}). In particular, they obtained the following low-entropy Schoenflies result:
\begin{theorem}[{Bernstein--Wang \cite{BernsteinWang:schoenflies}}]\label{theo:BW-schoenflies}
If $M^3\subset \RR^4$ has $\lambda(M) \leq \lambda(\SS^2\times \RR)$ then $M$ is smoothly isotopic to the round $\SS^3$. 
\end{theorem}
In \cite{BernsteinWang:schoenflies}, this is proven by flowing $M$ by mean curvature flow and then smoothing out any potential non-generic singularities to construct the desired isotopy. Our previous work \cite{CCMS:generic1} on generic mean curvature flow gave an alternative approach to this result by showing that if one perturbs $M$ slightly, the mean curvature flow directly \emph{provides} the isotopy:
\begin{theorem}[\cite{CCMS:generic1}]\label{theo:CCMS-low-ent-sch}
If $M^3\subset \RR^4$ has $\lambda(M) \leq \lambda(\SS^2\times \RR)$ then after a small $C^\infty$-perturbation to a nearby hypersurface $M'$, the mean curvature flow $M'(t)$ is completely smooth until it disappears in a round point. 
\end{theorem}

One of the consequences of this paper is a simplified proof of Theorem \ref{theo:CCMS-low-ent-sch} (see also the stronger version stated in Corollary \ref{coro:impr-schoenflies}). 

\subsection{Main results}
We now describe our main results in full generality. We construct generic mean curvature flows of sufficiently low-entropy hypersurfaces in all dimension. To quantify the low-entropy condition we make several definitions.\footnote{The definitions here are closely related to the hypotheses $(\star_{n,\Lambda}),(\star\star_{n,\Lambda})$ introduced by Bernstein--Wang (cf.\ \cite{BernsteinWang:topology-small-ent,BernsteinWang:schoenflies}), but our second hypothesis is less restrictive.} Let $\cS_n$ denote the set of smooth self-shrinkers in $\RR^{n+1}$ with $\lambda(\Sigma) < \infty$, i.e., properly embedded hypersurfaces $\Sigma$ satisfying $\bH + \frac{\bx^\perp}{2} = 0$ with finite Gaussian area. Let $\cS_n^*$ denote the non-flat elements of $\cS_n$. For $\Lambda > 0$, let 
\[
\cS_n(\Lambda) := \{\Sigma \in \cS_n : \lambda(\Sigma) < \Lambda\}, \qquad \cS_n^*(\Lambda) : = \cS_n(\Lambda) \cap \cS_n^*. 
\]
We also define
\[
\cS^\textrm{gen}_n : = \left\{O(\SS^{j}(\sqrt{2j}) \times \RR^{n-j}) \in \cS_n : j = 1,\dots,k, \, O \in O(n+1) \right\}
\]
to be the set of (round) self-shrinking spheres and cylinders in $\RR^{n+1}$. 

Similarly, we let $\cR\cM\cC_n$ denote the space of regular minimal cones in $\RR^{n+1}$, i.e., the set of $\cC \subset \RR^{n+1}$ with $\cC\setminus\{\bOh\}$ a smooth properly embedded hypersurface invariant under dilations and having vanishing mean curvature. Let $\cR\cM\cC_n^*$ denote the non-flat elements of $\cR\cM\cC_n$. Define 
\[
\cR\cM\cC_n(\Lambda) := \{\cC \in \cR\cM\cC_n : \lambda(\cC) < \Lambda\}, \quad \cR\cM\cC_n^*(\Lambda) : = \cR\cM\cC_n(\Lambda) \cap \cR\cM\cC_n^*. 
\]
For a dimension $n\geq 2$ and entropy bound $\Lambda \in (\lambda(\SS^n),2]$, our first hypothesis is
\[
\tag{$\dagger_{n,\Lambda}$} \textrm{For $3\leq k \leq n$, $\cR\cM\cC^*_k(\Lambda) = \emptyset$} 
\]
while our second hypothesis is 
\[
\tag{$\dagger\dagger_{n,\Lambda}$}  \cS_{n-1}^*(\Lambda) \subset  \cS^\textrm{gen}_{n-1} . 
\]

Finally, we define certain notation that will be used throughout.    
\begin{definition}\label{defi:ff-sing-gen}
For a closed embedded hypersurface $M^n\subset \RR^{n+1}$ we denote by $\fF(M)$ the set of cyclic\footnote{ Recall that a integral varifold $V$ is cyclic if the unique mod $2$ flat chain $[V]$ has $\partial [V] = 0$. Work of White \cite{White:cyclic} shows that this property is preserved under varifold (and Brakke flow) convergence.} unit-regular integral Brakke flows $\cM$ with $\cM(0) = \cH^n\lfloor M$, and for each $\cM \in \fF(M)$, we define $\sing_\textrm{gen} \cM \subset \sing\cM$ to be the set of singular points $(\bx,t)$ so that some\footnote{Note that if some tangent flow is a multiplicity one element of $\cS^{\textrm{gen}}_n$ then all are by \cite{ColdingIlmanenMinicozzi,ColdingMinicozzi:uniqueness-tangent-flow}, cf.\ \cite{BernsteinWang:high-mult-unique}.} tangent flow to $\cM$ at $(\bx,t)$ is a multiplicity-one flow associated to elements of $\cS^\textrm{gen}_n$.
\end{definition}
Having given these definitions, we can now state our main technical result. By convention we take $\lambda(\SS^0) = 2$. Everywhere below, $M$ is taken to be closed and embedded.
\begin{theorem}\label{theo:main-n-dim}
Assume that $n \geq 2$ and $\Lambda \in (\lambda(\SS^n),\lambda(\SS^{n-2})]$ satisfy hypothesis $(\dagger_{n,\Lambda})$ and $(\dagger\dagger_{n,\Lambda})$. If $M^n\subset \RR^{n+1}$ has $\lambda(M) \leq \Lambda$ then there exist arbitrarily small $C^\infty$ graphs $M'$ over $M$ so that $\lambda(M')< \Lambda$ and all $\cM' \in \fF(M')$ have $\sing\cM' = \sing_\textnormal{gen}\cM'$. In particular, the level set flow of $M'$ does not fatten. 
\end{theorem}

See \cite[Section 1.2]{CCMS:generic1} for a discussion of results related to the regularity of flows satisfying $\sing\cM' = \sing_\textrm{gen}\cM'$.

In low dimensions, the hypothesis $(\dagger_{n,\Lambda})$ and $(\dagger\dagger_{n,\Lambda})$ can be understood more concretely. This leads to the following results.

\begin{corollary}\label{coro:main-3d}
If $M^2\subset \RR^{3}$ has $\lambda(M) \leq 2$ then there exist arbitrarily small $C^\infty$ graphs $M'$ over $M$ so that the level-set flow  of $M'$ is non-fattening and the associated Brakke flow $\cM'\in \fF(M')$ has $\sing\cM' = \sing_\textnormal{gen}\cM'$. 
\end{corollary} 
\begin{proof}
Condition $(\dagger_{2,2})$ is vacuous while $(\dagger\dagger_{2,2})$ holds by the classification of self-shrinking curves \cite{AbreschLanger}. 
\end{proof}

\begin{corollary}\label{coro:main-4d}
If $M^3 \subset \RR^4$ has $\lambda(M) \leq \lambda(\SS^1\times \RR^2)$ then there exist arbitrarily small $C^\infty$ graphs $M'$ over $M$ so that the level-set flow of $M'$ is non-fattening and the associated Brakke flow  $\cM' \in \fF(M')$ has $\sing\cM' = \sing_\textnormal{gen}\cM'$.
\end{corollary}
\begin{proof}
By the resolution of the Willmore conjecture \cite{MarquesNeves}, $\cR\cM\cC^{*}_{3}(\Lambda_\cC) = \emptyset$ for 
\[
\Lambda_\cC = \frac{2\pi^2}{4\pi} \approx 1.57 > \lambda(\SS^1) \approx 1.52.
\]
Thus $(\dagger_{3,\Lambda})$ holds for all $\Lambda \leq \Lambda_\cC$. Furthermore, by the classification of low-entropy shrinkers in $\RR^3$ from \cite{BernsteinWang:TopologicalProperty}, it holds that $\cS^*_2(\lambda(\SS^1)) = \cS^\textrm{gen}_2$. Thus $(\dagger\dagger_{3,\lambda(\SS^1)})$ holds. 
\end{proof}

\subsection{Generic mean curvature flow with surgery}\label{subsec:surgery}
As already observed in \cite{CCMS:generic1}, we can apply Corollary \ref{coro:main-4d} to give a direct proof of Theorems  \ref{theo:BW-schoenflies} and \ref{theo:CCMS-low-ent-sch}. Moreover, Daniels-Holgate has recently proven that if an initial hypersurface admits a (cyclic, unit-regular, integral) Brakke flow with only\footnote{The spherical and neckpinch singularities are the tangent flows for which a canonical neighborhood theorem is proven, thanks to \cite{ChoiHaslhoferHershkovits,ChoiHaslhoferHershkovitsWhite}.} spherical and neckpinch type singularities\footnote{Note that if $\cM'$ is such a Brakke flow in $\RR^{n+1}$ and $\sing\cM'=\sing_{\textrm{gen}}\cM'$, then the condition ``$\cM'$ has only spherical and neckpinch singularities'' is a consequence of $\lambda(\cM') < \lambda(\SS^{n-2})$.} then it is possible to construct a smooth mean curvature flow with surgery starting from this initial condition (see \cite{Daniels-Holgate} for the precise definition of mean curvature flow with surgery).

As such, Corollaries \ref{coro:main-3d} and \ref{coro:main-4d} combined with \cite[Theorem 1.2]{Daniels-Holgate} yields the following generic surgery construction. 
\begin{corollary}[Generic mean curvature flow with surgery]
Assume that $n\geq 2$ and $\Lambda \in (\lambda(\SS^{n}),\lambda(\SS^{n-2})]$ satisfy $(\dagger_{n,\Lambda})$ and $(\dagger\dagger_{n,\Lambda})$. If $M^{n}\subset \RR^{n+1}$ has $\lambda(M)\leq\Lambda$, then there is an arbitrarily small $C^{\infty}$ graph $M'$ over $M$ and a smooth mean curvature flow with surgery starting from $M'$.
\end{corollary}

In particular, when $M^{3}\subset \RR^{4}$ is an embedded $3$-sphere with $\lambda(M) \leq \lambda(\SS^{1}\times \RR^{2})$, the mean curvature flow with surgery can be used (see \cite[Theorem 6.4]{Daniels-Holgate}) to construct an isotopy to the round $3$-sphere. This yields the strengthened version of the low-entropy Schoenflies theorem stated in Corollary \ref{coro:impr-schoenflies}.

\begin{remark} In the setting of $2$-convex mean curvature flow with surgery (see \cite{HuiskenSinnestrari:MCF-mean-convex,HuiskenSinnestrari:MCF-mean-convex,Brendle:inscribed-sharp,BrendleHuisken:R3,HaslhoferKleiner:estimates,HaslhoferKleiner:surgery,ADS,ADS2,BrendleChoi:3d,BrendleChoi:nD}) the surgery to isotopy construction has been studied in several works \cite{HuiskenSinestrari:surgery,BHH:2-convex-spheres,BuzanoHaslhoferHerskovits,Mramor:finiteness-surgery,MramorWang}. (We also mention related work using Ricci flow with surgery \cite{Marques:PSC,CarlottoLi} and singular Ricci flow \cite{BamlerKleiner:unique,BamlerKleiner:difeo,BamlerKleiner:contract}.) 
\end{remark}

\subsection{Generic regularity of area-minimizing hypersurfaces in eight dimensions} 

We remark that the study of generic mean curvature flow in our previous work \cite{CCMS:generic1} can be viewed as the parabolic analogue of the work of Hardt--Simon \cite{HardtSimon:foliation} and Smale \cite{Smale} concerning the generic regularity of area-minimizing hypersurfaces in eight dimensions. In particular, the existence and uniqueness of the ancient one-sided mean curvature flow \cite{CCMS:generic1} is a direct analogue of the existence and uniqueness of the foliation on either side of a regular area minimizing cone, as proven in \cite{HardtSimon:foliation} (see also \cite{Wang:smoothing}). 

In this paper, we develop a new technique based on density drop, that avoids the classification of the ancient one-sided flow. As one might expect, this also yields a new proof of the generic regularity results of Hardt--Simon \cite{HardtSimon:foliation} and Smale \cite{Smale} that avoids the need to classify the foliation. This is discussed further in Appendix \ref{app:HS}. 

\subsection{Organization}
See \cite[Section 2]{CCMS:generic1} for the conventions used in this paper. In Section \ref{sec:ent-drop} we prove entropy drop near non-generic singularities and we use this to prove Theorem \ref{theo:main-n-dim} in Section \ref{sec:proof-main}. Appendices \ref{app:stab-gen} and \ref{app:stab-cross} recall some standard stability results. Appendix \ref{app:local} contains a localized perturbative result. In Appendix \ref{app:HS}, we discuss how the arguments here relate to generic regularity of area-minimizing hypersurfaces in eight dimensions. 

\subsection{Acknowledgments} O.C.~was partially supported by a Sloan Fellowship, a Terman Fellowship, and NSF grants  DMS-1811059 and DMS-2016403.  K.C.~was supported by KIAS Individual Grant MG078901. C.M.~was supported by the NSF grant DMS-2050120 and DMS-2147521. F.S.~was supported by a Leverhulme Trust Research Project Grant RPG-2016-174.  We would like to thank Richard Bamler for some discussions related to weak flows and surgery constructions. Finally we are grateful to the referees for many helpful suggestions concerning

\section{Entropy drop near non-generic singularities}\label{sec:ent-drop}

\begin{lemma}\label{lemm:reg-shrinkers}
Assume that $(\dagger_{n,\Lambda})$ holds for some $\Lambda \leq 2$. Suppose that $V$ is a $F$-stationary cyclic integral $n$-varifold in $\RR^{n+1}$ satisfying $F(V) < \Lambda$. Then, there is $\Sigma \in \cS_n(\Lambda)$ so that $V = \cH^n\lfloor \Sigma$. 
\end{lemma}
\begin{proof}
This follows from the proof of \cite[Lemma 3.1 and Proposition 3.2]{BernsteinWang:topology-small-ent} except the cyclic property of $V$ is used to rule out three half-spaces as a potential iterated tangent cone (cf.\ \cite[Corollary 4.5]{White:cyclic}).
\end{proof}

Recall that Huisken has classified the cylinders $\SS^k(\sqrt{2k})\times \RR^{n-k}$ as the unique smooth embedded self-shrinkers with non-negative mean curvature $H \geq 0$ \cite{Huisken:sing,Huisken:local-global} (the technical assumption of bounded curvature was later removed by Colding--Minicozzi \cite{ColdingMinicozzi:generic}). The following result can be viewed as a geometric  consequence of Huisken's result. It will serve as our key mechanism for perturbing away ``non-generic'' singularities. 

\begin{proposition}\label{prop:geometric-generic}
For $\Sigma \in \cS_n^*$, fix an open set $\Omega\subset \RR^{n+1}$ with $\Sigma = \partial\Omega$. Assume that there is a space-time point $(\bx_0,t_0) \in (\RR^{n+1} \times \RR)\setminus (\bOh,0)$ so that
\begin{equation}\label{eq:geo-generic-assump}
\sqrt{t_0 -t}\, \Sigma + \bx_0 \subset \sqrt{-t} \, \bar \Omega
\end{equation}
for all $t < \min\{0,t_0\}$. Then, one of the following holds:
\begin{enumerate}
\item $\Sigma = \SS^n(\sqrt{2n})$, or
\item $\Sigma = O(\hat\Sigma \times \RR)$ for $\hat\Sigma \in \cS_{n-1}^*$ and $O \in O(n+1)$. 
\end{enumerate}
\end{proposition}
Note that if we replaced condition \eqref{eq:geo-generic-assump} with 
\begin{equation}\label{eq:geo-generic-assump-int}
\sqrt{t_0 -t}\, \Sigma + \bx_0 \subset \sqrt{-t} \, \Omega
\end{equation}
(i.e., if we replaced the closure of $\Omega$ with the interior of $\Omega$), we could use an inductive argument to conclude that $\Sigma \in \cS_n^\textrm{gen}$. 


Let us give the geometric intuition underlying our proof strategy. Let $\cM_0$ denote the spacetime track of $t \mapsto \sqrt{-t} \Sigma$ and $\cM$ denote the spacetime track of $t \mapsto \sqrt{t_0 -t}\, \Sigma + \bx_0$. For $\lambda \in (0, 1]$, let $\cM_\lambda$ be the parabolic rescaling of $\cM$ by a factor of $\lambda$; thus, $\cM_1 = \cM$ and, as $\lambda \to 0$, $\cM_\lambda \to \cM_0$ smoothly locally away from $(\bOh, 0)$. Note that $\cM_0$ is invariant under parabolic dilations, so $\cM_\lambda$ always lies weakly to one side of $\cM_0$.

If $\cM_\lambda$ touches $\cM_0$ for some $\lambda > 0$ (equivalently, for all $\lambda > 0$ due to $\cM_0$'s parabolic dilation invariance), it is then a simple consequence of the strong maximum principle and monotonicity that $\Sigma$ splits a line. 

Otherwise, $\cM_\lambda$ was disjoint from $\cM_0$ for all $\lambda \in (0, 1]$. It is then standard to use the height of $\cM_\lambda$ over $\cM_0$ at time $t=-1$, for $\lambda > 0$ small, to produce a kernel element of the linearized operator that is everywhere nonnegative ($\cM_\lambda$ always lies weakly to one side of $\cM_0$). By studying the geometry of parabolic dilations, the kernel element produced is $\bx_0 \cdot \nu_\Sigma$ if $\bx_0 \neq \bOh$ or $\bx \cdot \nu_\Sigma$ if $\bx_0 = \bOh$ ($\implies t_0 \neq 0$). It turns out that the former case implies splitting once again, while the latter implies the mean-convexity of $\Sigma$. 

The proof we give below is a more succinct version of the argument above: it handles both cases in a unified way.

\begin{proof}[Proof of Proposition \ref{prop:geometric-generic}]
Observe that the set $\cup_{t< 0} \sqrt{-t}\, \bar{\Omega} \times \{t\}$ is invariant under parabolic dilation around the space-time origin. We thus conclude that for all $\lambda \in [0,\infty)$ and $t < \min\{0,\lambda^2 t_0\}$, 
\[
\sqrt{\lambda^2t_0 -t}\, \Sigma +\lambda \bx_0 \subset \sqrt{-t} \, \bar\Omega
\]
In particular, taking $t=-1$ and $\lambda\geq0$ small, we have that 
\[
\lambda \mapsto \Sigma_\lambda : = \sqrt{1 + \lambda^2t_0}\, \Sigma +\lambda \bx_0 \subset  \bar\Omega
\]
is a $1$-parameter family of hypersurfaces with $\Sigma_0 = \Sigma = \partial\Omega$. The normal speed at $\lambda =0$ is $\bx_0 \cdot \nu_\Sigma \geq 0$ (where $\nu_\Sigma$ is the unit normal pointing into $\Omega$). Because 
\[
\Delta_\Sigma (\bx_0 \cdot \nu_\Sigma) - \tfrac 12 \bx \cdot \nabla_\Sigma (\bx_0 \cdot \nu_\Sigma) + |A_\Sigma|^2 (\bx_0\cdot\nu_\Sigma) = 0
\]
(cf.\ \cite[Theorem 5.2]{ColdingMinicozzi:generic}), the maximum principle implies that either $\bx_0\cdot\nu_\Sigma >0$ along $\Sigma$ or $\bx_0\cdot \nu_\Sigma = 0$ along $\Sigma$. (Note that $\Sigma$ is connected thanks to the Frankel property of shrinkers, cf.\ \cite[Corollary C.4]{CCMS:generic1}.)

In the first case (i.e., $\bx_0\cdot\nu_\Sigma >0$), each component of $\Sigma$ is a graph over the $\bx_0^\perp$-hyperplane. By \cite{Wang:graph} (cf.\ \cite{EckerHuisken:graphs}), each component of $\Sigma$ must be a hyperplane, so there is only one component and $\Sigma$ is a flat hyperplane. This contradicts the assumption that $\Sigma \in \cS_n^*$ (the set of non-flat shrinkers). 

In the second case (i.e., $\bx_0\cdot \nu_\Sigma = 0$), we see that $\bx_0\in T_p\Sigma$ for all $p\in \Sigma$. In particular, if $\bx_0\neq \bOh$, then $\Sigma$ splits a line in the $\bx_0$-direction. It thus remains to consider the situation in which $\bx_0=\bOh$. If this is the case, then it must hold that $t_0\neq 0$ and we have
\[
\tilde \Sigma_\mu := (1+\mu t_0) \Sigma \subset \bar\Omega. 
\]
 for $\mu\geq 0$ sufficiently small. The normal speed at $\mu=0$ is $t_0 \bx\cdot \nu_\Sigma \geq 0$. Using the shrinker equation, we thus find that $t_0 H_\Sigma \geq 0$. Since $t_0\neq 0$, we can assume that $H_\Sigma\geq 0$. Thus, up to a rotation, $\Sigma = \SS^{k}(\sqrt{2k})\times \RR^{n-k}$ for $k=1,\dots,n$ by \cite[Theorem 10.1]{ColdingMinicozzi:generic}. This completes the proof. 
\end{proof}

Recall the the definition of smoothly crossing Brakke flows in Definition \ref{defi:smooth-crossing}. 
\begin{proposition}\label{prop:density-drop-effective}
Fix $n \geq 2$, $\eps > 0$ and $\Lambda \in (\lambda(\SS^n),2]$ so that $(\dagger_{n,\Lambda})$ and $(\dagger\dagger_{n,\Lambda})$ hold. There is $\delta = \delta(n,\eps,\Lambda) > 0$ with the following property. 

Consider $\Sigma \in \cS_n^*(\Lambda-\eps) \setminus \cS_n^\textnormal{gen}$ and $\tilde \cM$  an ancient cyclic unit-regular integral $n$-dimensional Brakke flow in $\RR^{n+1}$ with $\lambda(\tilde \cM) \leq   F(\Sigma)$ so that $\tilde \cM$ does not smoothly cross the flow $(-\infty,0) \ni t\mapsto \cH^n\lfloor\sqrt{-t}\,\Sigma$. Then, $\Theta_{\tilde\cM}(\bx,t) \leq F(\Sigma) - \delta$ for all $(\bx,t) \in (\RR^{n+1}\times \RR)\setminus (\bOh,0)$. 
\end{proposition}
\begin{proof}
We argue by contradiction. Consider a sequence of $\Sigma_i \in \cS_n^*(\Lambda-\eps) \setminus \cS_n^\textnormal{gen}$ and $\cM_i$ ancient cyclic unit-regular integral Brakke flows in $\RR^{n+1}$ with $\lambda(\tilde \cM_i)\leq F(\Sigma)$ so that $\tilde \cM_i$ does not smoothly cross the flow $(-\infty,0) \ni t\mapsto \cH^n\lfloor\sqrt{-t}\,\Sigma_i$ and so that there are points $(\bx_i,t_i) \in (\RR^{n+1}\times \RR)\setminus (\bOh,0)$ with
\begin{equation}\label{eq:nearly-top-ent-assump}
\Theta_{\tilde\cM_i}(\bx_i,t_i) \geq F(\Sigma_i) - o(1)
\end{equation}
as $i\to\infty$. We can assume that $|(\bx_i,t_i)| = 1$.

By Lemma \ref{lemm:reg-shrinkers} and Allard's theorem \cite{Allard,Simon:GMT}, we can pass to a subsequence so that $\Sigma_i$ converges in $C^\infty_\textrm{loc}$ to $\Sigma \in \cS_n(\Lambda)$. By Brakke's theorem \cite{Brakke,White:Brakke}, $\Sigma$ is non-flat. Because cylinders are isolated in $C^\infty_\textrm{loc}$ by \cite{ColdingIlmanenMinicozzi}, we thus see that $\Sigma \in  \cS_n^*(\Lambda) \setminus \cS_n^\textnormal{gen}$. Note that $F(\Sigma_{i})\to F(\Sigma)$. 

We now pass to a further subsequence so that $(\bx_i,t_i) \to (\bx_0,t_0) \in \RR^{n+1}\times \RR$ with $|(\bx_0,t_0)| = 1$ and the Brakke flows $\tilde \cM_i$ converge to an ancient cyclic unit-regular integral Brakke flow $\tilde \cM$ with $\lambda(\tilde \cM) \leq F(\Sigma)$. By upper semi-continuity of Gaussian density, \eqref{eq:nearly-top-ent-assump} implies that $\Theta_{\tilde \cM}(\bx_0,t_0) \geq F(\Sigma)$. Because $\lambda(\tilde \cM) \leq F(\Sigma)$, $\tilde \cM$ is a self-similar flow around $(\bx_0,t_0)$. By stability of smoothly crossing flows, Lemma \ref{lemm:stab-crossing}, $\tilde\cM$ does not smoothly cross $(-\infty,0) \ni t\mapsto \cH^n\lfloor\sqrt{-t}\,\Sigma$.

Consider any tangent flow to $\tilde \cM$ at $t=-\infty$. By Huisken's monotonicity formula and Lemma \ref{lemm:reg-shrinkers} there is a smooth shrinker $\tilde\Sigma$ so that this tangent flow at $t=-\infty$ corresponds to some $\tilde\Sigma\in \cS_{n}(\Lambda)$ with multiplicity-one. By the Frankel property for self-shrinkers (cf.\ \cite[Corollary C.4]{CCMS:generic1}) and the strong maximum principle, if $\tilde\Sigma\neq\Sigma$ then the flows $t\mapsto \cH^n\lfloor \sqrt{-t}\,\Sigma$ and $t\mapsto \cH^n\lfloor \sqrt{-t}\, \tilde \Sigma$ smoothly cross each other at some point. This contradicts the stability of smooth crossings.

We conclude that any tangent flow to $\tilde \cM$ at $t=-\infty$ is the flow associated to $\Sigma$. Since $\tilde\cM$ is self-similar around $(\bx_{0},t_{0})$ we find
\[
\tilde \cM(t) = \cH^n\lfloor(\sqrt{t_0 - t}\, \Sigma + \bx_0)
\]
for $t < t_{0}$. Since $\tilde\cM$ does not smoothly cross $t\mapsto \cH^n\lfloor \sqrt{-t}\,\Sigma$, we see that there is an open set $\Omega \subset \RR^{n+1}$ with $\partial\Omega = \Sigma$ so that 
\[
\sqrt{t_0 - t}\, \Sigma + \bx_0 \subset \sqrt{-t} \, \bar\Omega
\]
for $t < \min\{0,t_{0}\}$. We can thus apply Proposition \ref{prop:geometric-generic} to conclude that (up to a rotation) $\Sigma = \hat \Sigma \times \RR$  for $\hat \Sigma \in \cS_{n-1}^*(\Lambda)$. By hypothesis $(\dagger\dagger_{n,\Lambda})$, $\hat\Sigma \in \cS^\textrm{gen}_{n-1}$ so $\Sigma = \hat\Sigma \times \RR \in \cS^\textrm{gen}_n$. This is a contradiction. 
\end{proof}

\section{Proof of Theorem \ref{theo:main-n-dim}}\label{sec:proof-main}

For $M'\subset \RR^{n+1}$ a smooth closed hypersurface, recall that $\fF(M')$ is the set of cyclic unit-regular integral Brakke flows $\cM'$ with $\cM'(0) = \cH^n\lfloor M'$. Note that \cite{Ilmanen:levelset,White:cyclic} implies that $\mathfrak{F}(M') \neq \emptyset$ (see also \cite[Appendix B]{HershkovtisWhite}). 

We define
\[
\cD(M') : = \sup\{ \Theta_{\cM'}(\bx,t) : \cM'\in\mathfrak{F}(M'), (\bx,t) \in \sing\cM'\setminus\sing_{\textrm{gen}}\cM'\}.
\]
Recall that by convention $\sup \emptyset = -\infty$.

Assume that hypotheses $(\dagger_{n,\Lambda})$ and $(\dagger\dagger_{n,\Lambda})$ hold for $\Lambda \in (\lambda(\SS^{n}),\lambda(\SS^{n-2}]$ fixed. Consider a smooth closed hypersurface $M^{n}\subset \RR^{n+1}$ with $\lambda(M) \leq \Lambda$. Flowing $M$ by mean curvature flow for a short time strictly decreases the entropy unless $M$ is homothetic to a self-shrinker. If $M$ is homothetic to a self-shrinker other than $\SS^{n}(\sqrt{2n})$ then by \cite{ColdingMinicozzi:generic}, a small $C^{\infty}$-perturbation of $M$ has strictly smaller entropy. 

As such, either $M=\SS^{n}(r)$ in which case the Theorem \ref{theo:main-n-dim} trivially holds or we can perform an initial perturbation and assume that $\lambda(M) \leq \Lambda - 2\eps$ for some $\eps>0$. Choose a foliation $\{M_{s}\}_{s\in (-1,1)}$ of a tubular neighborhood of $M$ so that $M_{0}=M$ and so that $\lambda(M_{s}) \leq \Lambda -\eps$. Fix $\delta=\delta(n,\eps,\Lambda)>0$ from Proposition \ref{prop:density-drop-effective}. 

\begin{lemma}\label{lemm:dens-drop}
We have
\[
\limsup_{s\to s_{0}} \cD(M_{s}) \leq \cD(M_{s_{0}}) - \delta.
\]
for all $s_{0}\in (-1,1)$. 
\end{lemma}

Lemma \ref{lemm:dens-drop} implies Theorem \ref{theo:main-n-dim} by a straightforward iteration argument since by Brakke's regularity theorem \cite{Brakke,White:Brakke}, if $\cD(M') \leq 1$ then $\cD(M') = -\infty$ implying that $\sing\cM' = \sing_{\textrm{gen}}\cM'$ for all $\cM'\in\mathfrak{F}(M')$. Since $\lambda(M') <  \Lambda \leq \lambda(\SS^{n-2}\times \RR^2)$, any $\cM'\in \mathfrak{F}(M')$ has only (multiplicity one) $\SS^n$ and $\SS^{n-1}\times\RR$-type singularities. Thus,  the resolution of the mean convex neighborhood conjecture for $\SS^{n-1} \times \RR$ singularities \cite{ChoiHaslhoferHershkovits,ChoiHaslhoferHershkovitsWhite} (cf.\ \cite{HershkovtisWhite}) implies non-fattening of the flow of $M'$.

\begin{proof}[Proof of Lemma \ref{lemm:dens-drop}]
Assume there is $s_{i}\to s_{0} \in (-1,1)$ with $s_i\neq s_0$ but
\[
\lim_{i\to\infty} \cD(M_{s_{i}}) > \cD(M_{s_{0}}) - \delta.
\]
Fix $\cM_{i} \in \mathfrak{F}(M_{s_{i}})$ and $(\bx_{i},t_{i})\in \sing\cM_{i}\setminus \sing_\textrm{gen}\cM_i$ with
\[
\lim_{i\to\infty}\Theta_{\cM_{i}}(\bx_{i},t_{i}) > \cD(M_{s_{0}}) - \delta
\]
Pass to a subsequence $\cM_{i}$ converging to $ \cM \in \mathfrak{F}(M_{s_{0}})$ and $(\bx_{i},t_{i})\to (\bx_0,t_0) \in \sing\cM$. Since $s_i\neq s_0$ for all $i$, we have that $M_{s_i}$ is disjoint from $M_{s_0}$ for all $i$. In particular, $\supp \cM_i \cap \supp\cM = \emptyset$ (by the avoidance principle for Brakke flows \cite[10.6]{Ilmanen:elliptic}). Thus, $(\bx_i,t_i) \neq (\bx_0,t_0)$. 

Observe that that if $(\bx_0,t_0)\in \sing_\textrm{gen}\cM$ then since $\lambda(M) < \Lambda \leq \lambda(\SS^{n-2})$, we see that $(\bx_0,t_0)$ must be a $\SS^n$ or $\SS^{n-1}\times \RR$-type singularity. Proposition \ref{prop:stable-sing} then implies that $(\bx_i,t_i) \in \sing_\textrm{gen}\cM_i$, a contradiction. Thus, it must hold that $(\bx_0,t_0) \in \sing\cM \setminus \sing_\textrm{gen}\cM$.


Translate $(\bx_0,t_0)$ to the space-time origin and parabolically dilate to yield $\tilde \cM_i$ and $(\tilde \bx_i,\tilde t_i)$ with $|(\tilde \bx_i,\tilde t_i)|=1$ and
\[
\lim_{i\to\infty}\Theta_{\tilde \cM_{i}}(\tilde \bx_{i},\tilde t_{i}) > \cD(M_{s_{0}}) - \delta
\]
Pass to a subsequence so that $\tilde \cM_i\rightharpoonup \tilde\cM$ and $(\tilde \bx_i,\tilde t_i) \to (\tilde \bx,\tilde t) \in (\RR^{n+1}\times \RR)\setminus(\bOh,0)$. By upper semicontinuity of density
\begin{equation}\label{eq:blow-up-density-is-big}
\Theta_{\tilde \cM}(\tilde \bx ,\tilde t) > \cD(M_{s_{0}}) - \delta.
\end{equation}
On the other hand, we can perform the same translation and parabolic dilation to $\cM$ and by extracting a further subsequence, the resulting flows converge to a tangent flow to $\cM$ at $(\bx_0,t_0)$. By Lemma \ref{lemm:reg-shrinkers}, the tangent flow is the multiplicity-one flow associated to a smooth shrinker $\Sigma$. Note that 
\[
F(\Sigma) \leq \lambda(\cM) \leq \limsup_{s\to s_0} \lambda(M_s) \leq \Lambda -\eps.
\]
Since $(\bx_0,t_0) \in \sing\cM\setminus\sing_\textrm{gen}\cM$ it must hold that $\Sigma \in \cS^*_n(\Lambda-\eps)\setminus\cS^\textrm{gen}_n$. Huisken's monotonicity formula implies that $\lambda(\tilde\cM) \leq F(\Sigma) =\Theta_\cM(\bx_0,t_0)$ (cf.\ the proof of Proposition 10.6 in \cite{CCMS:generic1}). Finally, since the supports of $\cM$ and $\cM_i$ are disjoint, $\cM_i$ does not smoothly cross $\cM$. As such (using Lemma \ref{lemm:stab-crossing}), $\tilde\cM$ does not smoothly cross $t\mapsto \cH^n\lfloor \sqrt{-t}\Sigma$. We can now apply Proposition \ref{prop:density-drop-effective} to conclude that 
\[
\Theta_{\tilde\cM}(\tilde \bx,\tilde t) \leq F(\Sigma) - \delta = \Theta_\cM(\bx_0,t_0) - \delta \leq \cD(M_{s_0}) - \delta. 
\]
This contradicts \eqref{eq:blow-up-density-is-big}, completing the proof. 
\end{proof}

\appendix

\section{Stability of generic singularities}\label{app:stab-gen}

Based on \cite{ChoiHaslhoferHershkovitsWhite}, the following stability of generic singularities was proven in \cite[Proposition 2.3]{SchulzeSesum} (see \cite[Lemma 10.4]{CCMS:generic1} for the simple argument when the singularity is modeled on $\SS^n$). When $n=2$ this also follows via density considerations using \cite{BernsteinWang:TopologicalProperty}. 
\begin{proposition}\label{prop:stable-sing}
Suppose that $\cM_i\rightharpoonup \cM$ are unit-regular integral Brakke flows in $\RR^{n+1}$ and that $(\bx_i,t_i) \in \sing\cM_i$ converge to $(\bOh,0) \in \sing_\textnormal{gen}\cM$. If the singularity at $(\bOh,0)$ is modeled on $\SS^n$ or $\SS^{n-1}\times \RR$, then for $i$ sufficiently large $(\bx_i,t_i) \in \sing_\textnormal{gen}\cM_i$. 
\end{proposition}

%

\section{Stability of crossing points} \label{app:stab-cross}

\begin{definition}\label{defi:smooth-crossing}
Given two integral unit Brakke flows $\cM^{(1)}$ and $\cM^{(2)}$, we say that $\cM^{(1)}$ and $\cM^{(2)}$ \emph{smoothly cross} at $(\bx,t)$ if there is $r>0$ with 
\[
\cM^{(j)}(s) \lfloor B_r(\bx) = \cH^n\lfloor \Gamma^{(j)}(s) 
\]
for $s \in (t-r^2,t+r^2)$ where $\Gamma^{(j)}(s)$ are smooth connected mean curvature flows so that in any small neighborhood of $\bx$ there are points of $\Gamma^{(1)}(0)$ on both sides of $\Gamma^{(2)}(0)$. 
\end{definition}

The following is a straightforward consequence of Brakke's regularity theorem \cite{Brakke,White:Brakke}.

\begin{lemma}\label{lemm:stab-crossing}
For $j=1,2$, suppose that $\cM_i^{(j)}\rightharpoonup \cM^{(j)}$ are integral unit-regular $n$-dimensional Brakke flows in $\RR^{n+1}$. Assume that $\cM^{(1)}$ smoothly crosses $\cM^{(2)}$ at $(\bx,t)$. Then, for $i$ sufficiently large, there is $(\bx_i,t_i)\to (\bx,t)$ so that $\cM^{(1)}_i$ smoothly crosses $\cM^{(2)}_i$ at $(\bx_i,t_i)$. 
\end{lemma}

\section{Local results}\label{app:local}

In this appendix we prove the following local perturbative result. 
\begin{proposition}\label{prop:local-result}
Suppose that $M^n\subset \RR^{n+1}$ is a closed embedded hypersurface, $\cM \in \cF(M)$ is a cyclic unit-regular integral Brakke flow starting at $M$. Assume that for $(\bx_0,t_0) \in \sing \cM$, the following holds:
\begin{itemize}
\item $\reg\cM \cap\{t<t_0\} \subset \RR^{n+1}\times \RR$ is connected
\item any tangent flow $\cN$ to $\cM$ at $(\bx_0,t_0)$ has $\cN(-1) = \cH^n\lfloor \Sigma$, for $\Sigma \in \cS_n^*\setminus \cS^\textnormal{gen}_{n}$ that does not split a line. 
\end{itemize}
Then, there is $r = r(\cM,\bx_0,t_0)>0$ so that for $M_j = \textnormal{graph}_M(u_j)$, $u_j>0$ with $u_j\to 0$ in $C^\infty$ it holds that any
\[
(\bx,t) \in B_r(\bx_0)\times (t_0-r^2,t_0+r^2)
\]
has $\Theta_{\cM_j}(\bx,t) \leq \Theta_\cM(\bx_0,t_0) - r$ for $j$ sufficiently large.
\end{proposition}
In particular, no tangent flow to $\cM$ at $(\bx_0,t_0)$ can arise as the tangent flow to $\cM_j$ at some point in $B_r(\bx_0)\times (t_0-r^2,t_0+r^2)$, for $j$ large. 

\begin{proof}
If this failed, there is $(\bx_j,t_j) \to (\bx_0,t_0)$ with
\[
\Theta_{\cM_j}(\bx,t) \geq \Theta_\cM(\bx_0,t_0) - o(1). 
\]
The assumption on the connectedness of the regular set implies that $\cM_j\lfloor \{t<t_0\}\rightharpoonup \cM\lfloor \{t<t_0\}$. Thus, by rescaling around $(\bx_0,t_0)$ so that $(\bx_j,t_j)$ is scaled to a unit distance from $(\bOh,0)$, we obtain $\Sigma \in \cS^*_n\setminus\cS^\textrm{gen}_n$ that does not split a line and an ancient Brakke flow $\tilde \cM$ that does not smoothly cross $t\mapsto\cH^n\lfloor\sqrt{-t}\, \Sigma$, so that $\lambda(\tilde\cM) \leq F(\Sigma)$, but for some $(\tilde \bx,\tilde t) \in (\RR^{n+1}\times \RR)\setminus\{(\bOh,0)\}$ it holds that $\Theta_{\tilde\cM}(\bx,t) \geq F(\Sigma)$. 

The argument in the second half of the proof of Proposition \ref{prop:density-drop-effective} carries over without change to show that there is an open set $\Omega\subset \RR^{n+1}$ with $\partial\Omega = \Sigma$ and 
\[
\sqrt{\tilde t - t}\, \Sigma + \tilde\bx \subset \sqrt{-t}\, \overline \Omega
\]
for $t<\min\{0,t_0\}$. By Proposition \ref{prop:geometric-generic}, we have that either $\Sigma = \SS^n(\sqrt{2n}) \in \cS^\textrm{gen}_n$ or $\Sigma$ splits a line.  Either case contradicts the assumption that $\cM$ has no such tangent flow at $(\bx_0,t_0)$. This completes the proof.
\end{proof}

Note that Proposition \ref{prop:local-result} does not give any indication as to \emph{how} the perturbation avoids the singularity (the trade-off is that the proof is very short). On the other hand, the results in \cite{CCMS:generic1} give a rather \emph{complete} description of how the perturbed flow avoids a compact/asymptotically conical singularity. The works \cite{SunXue:closed,SunXue:AC} also obtain some information along these lines, but only as long as the perturbed flow remains graphical over the original flow.

\section{The setting of area-minimizing hypersurfaces}\label{app:HS}

We recall the following fundamental result:
\begin{theorem}[{Hardt--Simon \cite[Theorem 2.1]{HardtSimon:foliation}}]\label{theo:HS-foliation}
If $\cC^n \subset \RR^{n+1}$ is a regular area minimizing cone then there exists smooth area-minimizing hypersurfaces $S_{\pm}$ in each component of $\RR^{n+1}\setminus\cC = U_+ \cup U_-$ so that if $S'$ is area minimizing and contained in $U_\pm$ then $S' = \lambda S_{\pm}$. 
\end{theorem}

The uniqueness statement in Theorem \ref{theo:HS-foliation} implies smoothness of solution to the Plateau problem for seven-dimensional currents in $\RR^8$ with generic boundary data (see \cite[Theorem 5.6]{HardtSimon:foliation}). Later, Smale used Theorem \ref{theo:HS-foliation} to prove that for $(M^8,g)$ a closed Riemannian manifold and $\alpha \in H_7(M;\ZZ)$, there is a $C^k$-close metric $g'$ so that the least area representative of $\alpha$ is smooth \cite{Smale}. 

\begin{remark} Besides their role in generic regularity of area-minimizing hypersurfaces in eight-dimensional manifolds, the surfaces $S_\pm$ are important objects in their own right, cf.\ \cite{IlmanenWhite:sharp.entropy,CLS,ZWang:deform,Li-ZWang,Simon:liouville,Simon:sing-arb}. In our previous paper \cite{CCMS:generic1}, we proved the parabolic analogue of Theorem \ref{theo:HS-foliation} (for compact/asymptotically-conical self-shrinkers) by constructing and classifying ancient one-sided flows analogous to the surfaces $S_{\pm}$. 
\end{remark}

We explain here how the main idea of this note can be used to prove the generic regularity results from \cite{HardtSimon:foliation,Smale} using the following result in lieu of Theorem \ref{theo:HS-foliation} (compare with Proposition \ref{prop:density-drop-effective}): 
\begin{proposition}
There is $\delta >0$ with the following property. Suppose that $\cC^7\subset \RR^8$ is a non-flat area-minimizing cone. If $S'$ is area-minimizing with support contained in $\bar{U}_{\pm}$, where $\RR^8 \setminus \cC= U_{+}\cup U_{-}$, then 
\[
\Theta_{S'}(\bx) \leq \Theta_\cC(0) - \delta.
\]
\end{proposition}
\begin{proof}
Using smooth compactness of the links of area minimizing cones in $\RR^8$ it suffices to rule out the case where $S'\subset \bar{U}_\pm$ is area-minimizing and there is $|\bx_0| = 1$ so that
\[
\Theta_{S'}(\bx_0) = \Theta_\cC(\bOh).
\]
Because $S'$ is contained in $\bar{U}_\pm$, its tangent cone at $\infty$ must be $\cC$ (e.g., using the Frankel property of minimal hypersurfaces in $\SS^n$). Thus, $S' = \cC + \bx_0$. This implies that $\cC + \lambda \bx_0 \subset \bar U_\pm$ as $\lambda \to0$, so $\bx_0 \cdot \nu_\cC \geq 0$. It cannot hold that $\bx_0\cdot \nu_\cC = 0$ since $\cC$ does not split a line, so $\bx_0\cdot \nu_\cC >0$. This would imply that $\cC$ is a graph, which is impossible since $\cC$ is non-flat.
\end{proof}

Using this, we obtain the following density drop result (compare with Lemma \ref{lemm:dens-drop}). 

\begin{corollary}
There is $\delta>0$ with the following property. Suppose that $\Sigma = \partial[\Omega] \subset B_2 \subset \RR^8$ is an area minimizing boundary with $\sing \Sigma = \{\bOh\}$. Suppose that $\Omega_1 , \Omega_2 ,\dots\supset \Omega$ is a sequence of sets of finite perimeter in $B_2$ with $\Sigma_i := \partial[\Omega_i]$ area minimizing, $\Sigma_{i}\cap\Sigma = \emptyset$, and $\Omega_i \to \Omega$. Then, for $\bx_i \in \Sigma_i \cap B_1$, we have
\[
\limsup_{i\to\infty} \Theta_{\Sigma_i}(\bx_i) \leq \Theta_\Sigma(\bOh) - \delta
\]
\end{corollary}
Note that this result can be iterated exactly in the proof of Theorem \ref{theo:main-n-dim} to obtain generic regularity of area-minimizing hypersurfaces in eight dimensions:
\begin{corollary}[cf.\ {\cite[Theorem 5.6]{HardtSimon:foliation}}]
For $\Gamma^{6}\subset \RR^{8}$ a smooth compact oriented submanifold without boundary, there is an arbitrarily small $C^{\infty}$-perturbation of $\Gamma$ to $\Gamma'$ so that any area-minimizing integral current bounded by $\Gamma'$ is completely smooth. 
\end{corollary}

\begin{corollary}[cf.\ {\cite{Smale}}]
For $(M^{8},g)$ a closed oriented Riemannian manifold and $\alpha \in H_{7}(M;\ZZ)$ a codimension-one integral homology class, there is an arbitrarily small $C^{k}$-perturbation of $g$ to $g'$ so that there is a unique $g'$-area-minimizing representative $\Sigma$ of $\alpha$ and $\Sigma$ is completely smooth. 
\end{corollary}

\bibliographystyle{alpha}
\bibliography{low-entropy-avoidance}

\end{document}